\documentclass[12pt,reqno]{amsart}

\usepackage{latexsym, amsmath, amssymb, amscd}

\usepackage{mathrsfs}

\newtheorem{theorem}{Theorem}[section]

\newtheorem{proposition}[theorem]{Proposition}

\theoremstyle{definition}
\newtheorem{definition}[theorem]{Definition}

\theoremstyle{remark}
\newtheorem{remark}[theorem]{Remark}

\numberwithin{equation}{section}

\newcommand{\ts}{\hspace{.11111em}}
\newcommand{\tts}{\hspace{.05555em}}

\DeclareMathOperator{\cat}{\operatorname{\mathsf{Cat}}\tts} %

\DeclareMathOperator{\G}{\operatorname{\mathcal{G}}}  

\DeclareMathOperator{\area}{\operatorname{\mathsf{Area}}\tts}
\DeclareMathOperator{\vol}{\operatorname{\mathsf{Vol}}\tts}
\DeclareMathOperator{\inj}{\operatorname{\mathsf{Inj}}\tts} 
\DeclareMathOperator{\emb}{\operatorname{\mathsf{Emb}}\tts} 
\DeclareMathOperator{\MinEnt}{\operatorname{\mathsf{Min Ent}}\tts} 

\DeclareMathOperator{\sys}{\operatorname{\mathsf{Sys}}\tts}

\DeclareMathOperator{\SR}{\operatorname{\mathsf{SR}}\tts}

\DeclareMathOperator{\dens}{\operatorname{\mathsf{Density}}\tts}

\begin{document}
 
\baselineskip.525cm

 
\title{Covering trick and embolic volume}

\author[L.~Chen]{Lizhi Chen}

\thanks{}

 \address{
\hspace*{0.055in}School of Mathematics and Statistics, Lanzhou University \newline
\hspace*{0.175in} Lanzhou 730000, P.R. China 
}

\email{\hspace*{0.025in} lizhi@ostatemail.okstate.edu}

\thanks{Supported by NSFC Grant No. 11901261 , and the Fundamental Research Funds for the Central Universities of No. lzujbky-2017-26 }



\date{\today}

\begin{abstract} 
 Embolic volume of compact manifolds is defined in terms of Berger's embolic inequality. In this paper, we show a result of relating embolic volume to the first Betti number. The proof relies on Gromov's covering argument appeared in systolic geometry. Berger called this method ``covering trick''. We exploit and present more details to covering trick in the paper.   
\end{abstract}

\maketitle 

\tableofcontents

\section{Introduction}

Main result of this paper is a relation between embolic volume and the first Betti number of compact Riemannian manifolds. The proof relies on Gromov's covering argument appeared in systolic geometry, see Gromov \cite[Section 5.3.]{Gromov1983}. In \cite[Section 7.2.1.2., Lemma 125]{Berger2003}, Berger named this method as ``covering trick''. In the paper, we further exploit and present more details to a refined version of covering trick (see \cite[Theorem 5.3.B.]{Gromov1983}).

Let $M$ be an $n$-dimensional manifold endowed with a Riemannian metric $\G$, denoted $(M, \G)$. Denote by $\inj (M, \G)$ the injectivity radius of $(M, \G)$. 
\begin{theorem}[Berger 1980, see Berger \cite{Berger1980} or {Berger \cite[Theorem 148]{Berger2003} }] \label{Berger1980_embolic_inequality}
 A compact $n$-dimensional Riemannian manifold $(M, \G)$ satisfies
 \begin{equation}\label{Berger_embolic}
  \inj(M, \G)^n \leqslant \alpha_n \vol_{\G}(M),
 \end{equation}
 where $\alpha_n$ is a constant only depending on the manifold dimension $n$. 
\end{theorem}
The constant $\alpha_n$ in (\ref{Berger_embolic}) given by Berger is $\frac{\sigma_n}{\pi^n}$, where $\sigma_n$ is the volume of unit Euclidean $n-$sphere. Our goal in this paper is to study the constant $\alpha_n$ in terms of topology. 
\begin{definition}
 The embolic volume of a compact $n$-dimensional manifold $M$, denoted $\emb (M)$, is defined as
 \begin{equation*}
  \inf_{\G} \frac{\vol_{\G}(M)}{\inj (M, \G)^n},
 \end{equation*}
 where the infimum is taken over all Riemannian metrics $\G$ on $M$. 
\end{definition}
\begin{remark}
 In Berger \cite{Berger2003}, $\emb (M)$ is called ``embolic constant''. In this paper, we use ``embolic volume'' appeared in Berger \cite[Section TOP.1.C., Page 102.]{Berger2000}. Other names of $\emb(M)$ include isoembolic volume (see \cite{Durumeric2002}) and optimal isoembolic constant ( see \cite{KaSa2005}).
\end{remark}

In terms of Berger's theorem (\ref{Berger1980_embolic_inequality}),
\begin{equation*}
 \emb (M) \geqslant \frac{\sigma_n}{\pi^n}
\end{equation*} 
holds for all compact $n$-manifolds $M$. Moreover, the equality holds for $n$-sphere $S^n$, and realized by standard Riemannian metric on $S^n$. The value $\frac{\sigma_n}{\pi^n}$ is isolated, that is, 
\begin{equation*}
 \emb(M) \geqslant \frac{\sigma_n}{\pi^n} + C
\end{equation*}
holds for some positive constant $C$ if $M$ is not $S^n$. 

In this paper, we show a relation concerning embolic volume and the first Betti number. Main result of this paper is as follows. 
\begin{theorem}\label{theorem_main}
 The embolic volume has the following lower bound in terms of the first Betti number $b_1 (M)$,
 \begin{equation}\label{embolic}
  \emb (M) \geqslant C_n \frac{b_1 (M)}{\exp{ \left( C_n^{\prime} \sqrt{b_1 (M)} \right) }} ,
 \end{equation}
 where $C_n$ and $C_n^{\prime}$ are two constants only depending on the manifold dimension $n$, and can be explicitly given. 
\end{theorem} 
\begin{remark}
 The inequality (\ref{embolic}) yields
 \begin{equation*}
  \emb (M) \geqslant C_{\theta} \ts \tts b_1(M)^{1 - \theta}
 \end{equation*}
 for any positive constant $\theta$, where $C_{\theta}$ is a constant only depending on $\theta$. 
\end{remark}

Theorem \ref{theorem_main} is proved by using a refined version of the covering trick. Briefly speaking, the covering trick is to construct a maximal system of disjoint balls. Then the system of balls with the same centers and double radii is a covering space to the manifold.

Another result contained in the proof of Theorem \ref{theorem_main} is an upper bound to the number of contractible balls need to cover the manifold.
\begin{theorem}\label{theorem_cover}
 Let $N$ be the least number of contractible balls in an open covering of a compact $n$-dimensional manifold $M$. There exists an upper bound of $N$ in terms of the embolic volume,
 \begin{equation}\label{ball_embolic}
  N \leqslant D_n \ts \emb(M) \ts 5^{ (n + 1) \sqrt{ \log_5{ \emb(M) } } },
 \end{equation}
 where $D_n$ is a constant only depending on the manifold dimension $n$. 
\end{theorem}

The estimate (\ref{embolic}) in Theorem \ref{theorem_main} is an extension to Durumeric's earlier result.
\begin{theorem}[Durumeric 1989, see \cite{Durumeric1989}]
 The embolic volume $\emb(M)$ of a compact $n$-dimensional manifold $M$ satisfies
 \begin{equation}\label{Durumeric}
  \emb(M) \geqslant \sqrt{b_1 (M)}. 
 \end{equation}
\end{theorem}
\begin{remark}
 In \cite{Durumeric1989}, Durumeric showed the estimate (\ref{Durumeric}) in a different method. It is not hard to see that the covering trick also can be used to get an alternative proof.  
\end{remark} 
Moreover, Durumeric related embolic volume to the number of generators of fundamental group, see \cite{Durumeric2002}. 

Proved by Yamaguchi in \cite{Yamaguchi1988}, for any constant $C > 0$, there are only finitely many homotopy types contained in the set $\{ M | \emb(M) < C \}$ of compact manifolds $M$. Furthermore, set of embolic volume is discrete. 
\begin{theorem}[Grove, Petersen and Wu 1990, see \cite{GrovePW1990} or Berger \cite{Berger2003}]
 The set of embolic volume $\emb (M)$ is discrete when $M$ runs through all manifolds of a given dimension $n$. 
\end{theorem}
Another result of embolic volume in earlier years is Croke's lower bound in terms of manifold category. For a compact manifold $M$, the category is defined to be the minimum number of topological balls needed to cover it, denoted $\cat (M)$. 
\begin{theorem}[Croke 1988, see \cite{Croke1988}]
Embolic volume of compact $n$-dimensional manifolds satisfies
\begin{equation*}
 \emb (M) \geqslant \emb(S^n) +  c(n) \left( \cat (M) - 2 \right),
\end{equation*}
where $c(n)$ is a constant only depending on the manifold dimension $n$. 
\end{theorem}

In \cite[Section 11.2.3.]{Berger2003}, Berger summarized:  `` $\emb(M)$ is quite a good functional for classifying compact manifolds '', and `` the embolic volume appears as if it is measuring the complexity of $M$ ''. However, what we know about the embolic volume is still very limited. Open problems of embolic volume can be found in Berger \cite[Section 11.2.3.]{Berger2003}. 

Later progress of embolic volume include relation to minimal volume entropy and simplicial volume. 
\begin{theorem}[Katz and Sabourau 2005, see \cite{KaSa2005}] 
 Let $M$ be a compact $n$-dimensional manifold $M$. 
 Denote by $\MinEnt(M)$ the minimal volume entropy, $\| M \|$ the simplicial volume. Then the following two inequality hold, 
  \begin{align*}
   & \emb (M) \geqslant \lambda_n \ts \frac{\MinEnt(M)^n}{\log^n{\left( 1 + \MinEnt (M) \right)}} , \\
   & \\
   & \emb (M) \geqslant \lambda_n^{\prime} \ts \frac{\| M \|}{\log^n{\| M \|}} , 
  \end{align*}
  where $\lambda_n$ and $\lambda_n^{\prime}$ are two positive constants only depending on the manifold dimension $n$. 
\end{theorem}
An application of embolic volume in physics can be found in \cite{Kalogeropoulos2018}. 

Organization of paper: In Section 2, covering trick and systolic geometry is introduced. In Section 3, we show the proof of Theorem \ref{theorem_main}. In the proof, a detailed analysis to the refined version of covering trick is presented. Gromov has used this refined version of covering trick to relate systolic volume and the first Betti number. Theorem \ref{theorem_cover} is proved in Section 4. 

\section{Covering trick and systolic geometry}

Let $X$ be a metric space. The covering trick is summarized by Berger in \cite[Section 7.2.1.2, Lemma 125]{Berger2003}. Main steps of this method include:
\begin{enumerate}
 \item Construct a maximal system $\{B(x, r)\}$ of balls with radius $r$ in $X$, so that balls with larger radius $r^{\prime} > r$ will have some nonempty intersections;
 \item Keep the same centers, but let radii be doubled, then the system $\{ B(x, 2r) \}$ of balls is a cover of $X$. 
\end{enumerate}

There are applications of the covering trick in systolic geometry. Gromov showed that systolic volume of hyperbolic surfaces with large genus depends on the first Betti number (thus depends on genus). In systolic geometry, homotopy $1$-systole of a closed Riemannian manifold $(M, \G)$ is defined to be the shortest length of a noncontractible loop, denoted by $\sys \pi_1 (M, \G)$. For a non-simply connected closed surface $X$, systolic inequality
\begin{equation*}
 \sys \pi_1 (X, \G)^2 \leqslant C \ts \area_{\G}(X)
\end{equation*}
holds for any Riemannian metric $\G$ on $X$ ( see Katz \cite{Katz2007}), where $C$ is a positive constant. Analogously, we define systolic volume of a non-simply connected surface $X$ as
\begin{equation*}
 \SR(X) = \inf_{\G} \frac{\area_{\G}(X)}{\sys \pi_1 (X, \G)^2} ,
\end{equation*}
where the infimum is taken over all Riemannian metrics $\G$ on $X$. In \cite[Proposition 5.2.B.]{Gromov1983}, Gromov showed that $\SR(X) \geqslant \frac{3}{4}$ holds for all compact surfaces $X$ with infinite fundamental group. Moreover, for a closed hyperbolic surface $X$ with genus $g$, in \cite[Proposition 5.3.A.]{Gromov1983} Gromov proved
\begin{equation}\label{SR_1}
 \SR(\Sigma) \geqslant A \ts \sqrt{b_1 (X)}
\end{equation} 
by using covering trick, where $A$ is a positive constant. Let $\Gamma = \gamma_1 \cup \gamma_2 \cup \cdots \cup \gamma_b $ be a system of loops representing homology basis of $H_1(X)$. Main idea of Gromov's proof of (\ref{SR_1}) is to apply covering trick by using open balls $B(x_j, r_j)$ with centers $x_j \in \Gamma$. In particular, each ball $B(x_j, r_j)$ has radius $r_j = \frac{\sys \pi_1 (X, \G)}{4}$, so that
\begin{equation*}
 \area_{\G} (B(x_j, r_j)) \geqslant 2 \ts r_j^2 = \frac{1}{8} \sys \pi_1 (X, \G)^2. 
\end{equation*}
Let $N$ stand for the number of balls $B(x_j, r_j)$. The first Betti number $b_1 (X)$ satisfies
\begin{equation*}
 b_1 (X) \leqslant \frac{N (N - 1)}{2}, 
\end{equation*}
since each closed geodesic $\gamma_j$ is covered by at least two balls $B(x_j, r_j)$. Hence the estimate (\ref{SR_1}) holds. 
\begin{remark}
 For hyperbolic surfaces $\Sigma_g$ with genus $g$, when $g$ is large, 
 Gromov proved that optimal estimate of systolic volume is
 \begin{equation*}
  \SR(\Sigma_g) \geqslant C \ts \frac{g}{( \log{g} )^2},
 \end{equation*}
 where $C$ is a fixed constant, see Berger \cite{Berger2003} or Katz and Sabourau \cite{KaSa2005}. 
\end{remark}

By a refined version of covering trick, Gromov showed the following result of systolic volume of surfaces.
\begin{theorem}[Gromov 1983, see \cite{Gromov1983}]
 The systolic volume of a hyperbolic surface $X$ satisfies
 \begin{equation}
  \SR(X) \geqslant \delta \ts \frac{b_1 (X)}{ \exp{ \left( \delta^{\prime} \sqrt{b_1 (X) } \right) } } ,
 \end{equation}
 where $\delta$ and $\delta^{\prime}$ are two constants. 
\end{theorem}

In this paper, we apply covering trick of the refined version to get an estimate of embolic volume, see Section 3. In particular, some details are added for a better understanding of the method, and this is the original motivation to write this short note.  

\section{Proof of Theorem \ref{theorem_main}}

\subsection{Volume growth of balls}

In a Riemannian manifold $(M, \G)$, volume of balls with small radius is related to scalar curvature. Let $Sc_p$ be the scalar curvature of a point $p \in M$. 
\begin{theorem}[ see {\cite[Chapter III, H.]{GallotHL2004}}] \label{Croke_balls}
 Let $B(p, r)$ be a ball with center $p$ and radius $r$ in an $n$-dimensional Riemannian manifold $(M, \G)$. If the radius $r$ is small, then the volume of $B(p, r)$ satisfies
 \begin{equation}
  \vol_{\G} (B(p, r)) = r^n \omega_n \left[ 1 -  \frac{\text{Sc}_p}{6 (n + 2)} + o(r^2) , \right]
 \end{equation}
 where $\omega_n$ is the volume of standard Euclidean unit $n$-ball. 
\end{theorem} 

We show an estimate of growth of volume of balls. The following proposition is a slight generalization of Lemma 1 in Guth \cite[Section 1.]{Guth2011}. 
\begin{proposition}
 Let $p \in M$ be any point, and $R_0 > 0$ be a given value. For any positive fixed value $\theta$, there exists at least one $R \in (0, R_0]$, such that 
 \begin{equation}\label{ball_growth}
  \vol_{\G} (B(p, 5 R)) \leqslant \alpha \vol_{\G} ( B(p, R) ) 
 \end{equation}
holds for $\alpha = 5^{n + \theta}$. 
\end{proposition}
\begin{proof} 
 For the given point $p \in M$, we define a density function with radius $r$ to be
 \begin{equation*}
  \dens_p ( r ) =  \frac{ \vol (B(p, r)) }{r^n} . 
 \end{equation*}
 Due to the local volume identity (\ref{Croke_balls}), we have the following fact,
 \begin{equation*}
  \lim_{r \to 0} \dens_p (r) = \omega_n .
 \end{equation*}
Hence if we assume that the estimate (\ref{ball_growth}) is not true, then for the given value $\alpha = 5^{n + \theta}$,
 \begin{align*}
  \vol_{\G} (B(p, 5R)) > \alpha \vol_{\G} ( B(p, R) ) 
 \end{align*}
 holds for any $0< R \leqslant R_0$, so that
 \begin{equation*}
  \dens_p (5R) > 5^{\theta} \dens_p (R) .
 \end{equation*}
 We have
 \begin{align*}
  \dens_p (5R_0) & > 5^{\theta} \dens_p (R_0) \\ 
   & > 5^{2 \theta} \dens_p (5^{-1} R_0 ) \\
   & > 5^{3 \theta} \dens_p ( 5^{-2} R_0 ) \\
   & \hspace{20pt} \vdots \\
   & \hspace{20pt} \vdots \\
   & > 5^{ \ell \ts \theta } \dens_p ( 5^{-\ell + 1} R_0 ) .
 \end{align*}
 However, this is a contradiction. If we let $\ell \to \infty$, the right side of 
 \begin{equation*}
  \dens_p(5 \tts R_0)  > 5^{\ell \tts \theta} \dens_p ( 5^{- \ell + 1} R_0 ) ,
 \end{equation*}
 is going to infinity, since $\theta > 0$ and $\displaystyle \lim_{\ell \to \infty} \dens_p (5^{-\ell + 1} R_0 ) = \omega_n$. 
\end{proof}

\vskip 20pt

\subsection{Admissible balls}

Now assume that
\begin{equation*}
 R = \sup \left\{ r \left| \ts 0 < r \leqslant R_0 \; \text{and} \; \vol_{\G} (B(p, 5r)) \leqslant \alpha \vol_{\G} ( B(p, r) ) \right. \right\} .
\end{equation*}
If $R < R_0$, then there will be a positive integer $k \; (k \geqslant 1)$ such that
\[ 5^{-k} R_0 \leqslant R < 5^{-k + 1} R_0 .  \] 
When $R = R_0$, we can set $k = 0$. The integer $k$ thus depends on point $p$. 


Balls $B(p, R)$ are called ``admissible balls'' in Gromov \cite[Section 5.3.]{Gromov1983}. 
\begin{definition}
 A ball $B(p, R) \subset M$ with center $p$ and radius $R \in (0, R_0]$ is called an admissible ball, if either one of the following two holds,
 \begin{enumerate}
  \item $R = R_0$ and
    \[ \vol_{\G} B(p, 5R_0) \leqslant \alpha \vol_{\G} ( B( R_0 ) ) , \]
  \item $0 < R < R_0$ and the following two conditions hold,
   \begin{enumerate}
    \item $ \displaystyle \vol_{\G} B(p, 5R ) \leqslant \alpha \vol_{\G} (B(R)), $
    \item $ \vol_{\G} ( B(p, 5R^{\prime}) ) > \alpha \vol_{\G} ( B(p, R^{\prime}) ) $ holds for any $R^{\prime} \in (R, R_0 ] $.
  \end{enumerate} 
 \end{enumerate}
\end{definition}

Next we show an upper bound of the value of $k$. A local version of Berger's embolic inequality is proved by Croke. We will use this local estimate. 
\begin{theorem}[Croke 1980, see Croke \cite{Croke1980} or Berger {\cite[Theorem 149]{Berger2003}}]
 Let $(M, \G)$ be a compact $n$-dimensional Riemannian manifold. For any point $p \in M$, the metric ball $B(p, R)$ with center $p$ and radius $0< R \leqslant \frac{1}{2} \inj (M, \G)$ satisfies
 \begin{equation}
  \vol_{\G} (B(p, R)) \geqslant \beta_n R^n,
 \end{equation}
 where $\beta_n$ is a constant only depending on the manifold dimension $n$. 
\end{theorem} 
\begin{remark}
 A non-optimal value of $\beta_n$ is given in Croke \cite{Croke1980} (also see Berger \cite{Berger2003}),
 \begin{equation*}
  \beta_n = \frac{2^{n-1} \ts \sigma_{n-1}^n}{n^n \ts \sigma_n^{n-1}},
 \end{equation*}
 recall that $\sigma_n$ stands for the volume of unit Euclidean $n$-sphere.
\end{remark}
We choose $R_0 = \frac{1}{2} \inj (M, \G) $. Suppose that for an admissible ball $B(p, R) \subset M$, the radius $R \in (0, R_0)$, then we have the following estimates,
\begin{align*}
 \vol_{\G} (M) & \geqslant \vol_{\G} (B(p, 5R_0))  \\ 
  & > \alpha \vol_{\G} ( B(p, R_0) ) \\ 
  & > \alpha^2 \vol_{\G} ( B(p, 5^{-1} R_0) ) \\ 
  & \hspace{20pt} \vdots \\ 
  & > \alpha^k \vol_{\G} ( B(p, 5^{-k + 1} R_0 ) ) \\
  & > \alpha^k \vol_{\G} ( B(p, R) ) \\
  & > \alpha^k \ts \beta_n R^n \\
  & > \alpha^k \beta_n ( 5^{-k} R_0 )^n , 
\end{align*}
so that
\begin{equation}\label{integer_k}
 k < \frac{ \log_5{\frac{ \vol_{\G} (M) }{\beta_n R_0^n} }  }{-n + \log_5{\alpha}} . 
\end{equation}

\vskip 20pt

\subsection{Refined version of covering trick}

Let $\Gamma = \gamma_1 \cup \gamma_2 \cdots \cup \gamma_b$, with $\{ \gamma_1, \gamma_2 , \cdots , \gamma_b \}$ a system of loops representing the homology basis of $H_1 (M)$. Note that the first Betti number $b_1 (M)$ is equal to $b$. We choose a maximal system of balls $\{ B(p_j, R_j) \}_{j = 1}^N$ with centers located on $\Gamma$. Then in terms of the maximality, $\{ B(p, 2R_j) \}_{j=1}^N$ covers the graph $\Gamma$. And we have the following relation between the first Betti number $b_1 (M)$ and the number $N$ of balls,
\begin{equation*}
 b_1 (M) \leqslant T - N + C ,
\end{equation*}
where $T$ is the number of pairwise intersections in the cover $\{ B(p_j, 2R_j) \}_{j=1}^N$, and $C$ is the number of components of the open cover $\{ B(p_j, 2R_j) \}_{j = 1}^N$. It is easy to see that $N \geqslant C$, so that $T$ is an upper bound of $b_1(M)$. 

The number of pairwise intersections in $\{ B(p_j, 2R_j) \}_{j = 1}^N$ can be estimated in terms of the injectivity radius. One fact is that if $B(p_{j_{i}}, 2R_{j_{i}}) \cap B(p_{j_{\ell}} , 2R_{j_{\ell}} ) \neq \emptyset  $, and $R_{j_{\ell}} \leqslant R_{j_{i}}$, then we must have
\begin{equation}
 B( p_{j_{\ell}} , R_{j_{\ell}} ) \subset B(p_{j_{i}} , 5 R_{j_{i}} ).
\end{equation}
Assume that the ball $B(p_j, 2R_j)$ has nonempty intersections with balls in $\{ B(p_{j_1}, 2R_{j_1}), B(p_{j_2}, 2R_{j_2}), \cdots , B(p_{j_{T_j}}, 2R_{j_{T_j}}) \}$, and $R_{j_i} \leqslant R_{j}$.   
Let $\displaystyle \alpha = 5^{ n + \sqrt{\log_5{\frac{V}{\beta_n R_0^n}} } }$. So we choose $\alpha = 5^{n + \theta}$ with $\theta = \sqrt{ \log_5{ \frac{V}{\beta_n \ts R_0^{n}} } }$. Then we have
\begin{align*}
 V & = \vol_{\G}(M) \\
  & > \sum_{j = 1}^N \vol_{\G}( B(p_j, R_j) ) \\ 
  & = \alpha^{-1} \sum_{j = 1}^N \alpha \vol_{\G} ( B(p_j , R_j) ) \\
  & \geqslant \alpha^{-1} \sum_{j=1}^N \vol_{\G} ( B(p_j, 5R_j) ) \\
  & \geqslant \alpha^{-1} \sum_{j=1}^N \sum_{i = 1}^{T_j} \vol_{\G} ( B( p_{j_i}, R_{j_i} ) )  \\
  & \geqslant \alpha^{-1} \sum_{j = 1}^N \sum_{i = 1}^{T_j} \beta_n R_{j_i}^n \\
  & \geqslant \alpha^{-1} \sum_{j = 1}^N \sum_{i = 1}^{T_j} \beta_n \left( 5^{-k_{i, j}} R_0 \right)^n \\ 
  & \geqslant \alpha^{-1} \sum_{j=1}^N \sum_{i = 1}^{T_j} \beta_n 5^{ - n \sqrt{ \log_5{ \frac{V}{\beta_n R_0^n} } } } R_0^n \\
  & = \alpha^{-1} \ts T \ts \beta_n 5^{ - n \sqrt{ \log_5{ \frac{V}{\beta_n R_0^n} } } } R_0^n , 
\end{align*}
where the last inequality holds since the upper bound of $k_{i, j}$ showed in (\ref{integer_k}), and recall that $T$ is equal to the total number of pairwise intersections. Hence we have
\begin{align*}
 T & \leqslant \frac{V}{\beta_n R_0^n} 5^{  n \sqrt{ \log_5{ \frac{V}{\beta_n R_0^n}  } } } \cdot \alpha \\
     & \leqslant \frac{V}{\beta_n R_0^n} 5^{n +  (n + 1)\sqrt{ \log_5 { \frac{V}{\beta_n R_0^n} } }}  . 
\end{align*}
Let $R_0 = \frac{1}{2} \inj (M, \G) $. Hence the lower bound in terms of the first Betti number is obtained,
\begin{equation}
  \frac{\vol_{\G}(M)}{\inj(M, \G)^n} \geqslant C_n \ts \frac{b_1(M)}{\exp{ \left( C_n^{\prime} \sqrt{ \log{b_1(M)} } \right) }},  
\end{equation}
where $C_n$ and $C_n^{\prime}$ are two constants only depending on the manifold dimension $n$. Finally we obtain the estimate (\ref{embolic}) in Theorem \ref{theorem_main}.
\begin{remark}
 Combined with the value of $\beta_n$ given in Croke's estimate, we can explicitly provide values of $C_n$ and $C_n^{\prime}$ as follows,
 \begin{equation*}
  C_n = \frac{5 \tts \sigma_{n-1}^n}{ 2 \tts n^n \sigma_n^{n-1}} , \quad C_n^{\prime} = (n + 1) \sqrt{\log{5}}. 
 \end{equation*}
\end{remark}

\section{Proof of Theorem \ref{theorem_cover}}

In the proof of Theorem \ref{theorem_main}, we verified that the number $T$ of pairwise intersections in the cover $\{ B(p_j, 2 R_j) \}_{j = 1}^N$ has the following upper bound,
\begin{equation*}
 T \leqslant \frac{V}{\beta_n \ts R_0^n} \ts 5^{n + (n+1) \sqrt{ \log_5 { \frac{V}{\beta_n R_0^n } } } },
\end{equation*}
so after taking an infimum over all Riemannian metrics $\G$ on $M$, we have
\[ T \leqslant D_n \ts \emb(M) \ts 5^{ (n+1) \sqrt{ \log_5{\emb(M) } } } , \]
where $D_n$ is a constant only depending on the manifold dimension $n$.

\end{document}